\documentclass[a4paper,12pt]{article}
\usepackage{glaudo_en_pkg}

	\newcommand{\Eps}{\mathcal{E}}

\DeclareMathOperator{\Aut}{Aut}
\DeclareMathOperator{\SL}{SL}

\begin{document}

	\title{Non-rational varieties with the Hilbert Property}
	\author{Julian Lawrence Demeio}
	
%
	\maketitle
	
		\begin{abstract}
		A variety $X$ over a field $k$ is said to have the Hilbert Property if $X(k)$ is not thin. We shall exhibit some examples of varieties, for which the Hilbert Property is a new result.

We give a sufficient 
condition for descending 
the Hilbert Property to the quotient
of a variety by the 
action of a finite group.
Applying this result
to linear actions of groups,
we exhibit some examples
of non-rational unirational
varieties with the Hilbert Property, 
providing positive instances 
of a conjecture posed by
Colliot-Thélène and Sansuc.

We also give a sufficient
condition for a surface with
two elliptic fibrations to
have the Hilbert Property, and use it to 
prove that a certain class
of K3 surfaces have the Hilbert Property,
generalizing a result of
Corvaja e Zannier.
		 
		 
	\end{abstract}

	\section{Introduction}
	
	
	In 1917, Emily Noether tried to approach the Inverse Galois Problem by using the Hilbert Irreducibility Theorem \cite[Thm. 3.4.1]{Serre}. Her strategy relied on her rationality conjecture that every field of the form $k(x_1,\dots,x_k)^G$, $G$ being a finite group acting by permutation on the variables $x_i$, is rational. It is now well known that this conjecture is false \cite{swan}.
	
	The Hilbert Property was introduced by Colliot-Thélène and Sansuc in \cite{congetturastrong}, in an effort to try to recover the original strategy of Noether, by generalizing the classical Hilbert Irreducibility Theorem. 
	
	A geometrically irreducible variety $X$ over a field $k$ is said to have the \textit{Hilbert Property} if, for any finite morphism $\pi:E \rightarrow X$, such that $X(k)\setminus \pi(E(k))$ is not Zariski-dense in $X$, there exists a rational section of $\pi$ (see  \cite[Ch. 3]{Serre} for an introduction of the Hilbert Property).
	The Hilbert Irreducibility Theorem can then be reformulated by saying that rational varieties over number fields have the Hilbert Property.
	
	\smallskip
	
	Colliot-Thélène and Sansuc \cite{congetturastrong} raised the following question:
	
	\begin{question}\label{congetturona}
		Do unirational varieties defined over number fields have the Hilbert Property?
	\end{question}
	
	They observed that a positive answer to Question \ref{congetturona} would settle the Inverse Galois Problem. Soon after, Colliot-Thélène conjectured (see \cite[Sec. 3.5]{Serre}) that every unirational variety has actually a stronger property, i.e. the weak weak approximation property. This conjecture would imply, in particular, a positive answer to Question \ref{congetturona}.

	So far, progress on Question \ref{congetturona} has mainly been achieved in the context of the weak weak approximation for some quotient varieties of the form $G/H$, where $G$ is a linear group and $H$ is a subgroup. Namely, the weak weak approximation property for $G/H$ has been proven when $G$ is semisimple simply connected  and $H$ is either connected \cite{borovoi} or its group of connected components is particularly simple (for instance a semidirect product of abelian groups \cite{harari}, or a solvable group satisfying a certain coprimality condition \cite{weakweaksolvable}). Moreover, when $H$ is a finite group,
	the weak weak approximation and the Hilbert Property of the variety $\SL_n/H$ are equivalent, respectively, to the weak weak approximation and the Hilbert Property of the variety $\A^n/H$\footnote{This follows from the no-name lemma of Bogomolov and Katsylo \cite{bogKats}.}.
	


%
%
%
%
%
%

 Although a lot of work has been devoted to the weak weak approximation, some recent developments concerning the Hilbert Property show that it manifests more geometric flexibility.  In \cite{fibrazioniHP} Bary-Soroker, Fehm, and Petersen show that, under some mild natural hypothesis, some fibrations have the Hilbert Property. In \cite{articoloHP}, using multiple elliptic fibrations, Corvaja and Zannier proved that the Fermat surface $F:\{x^4+y^4=z^4+w^4\}\subset \P_3$ has the Hilbert Property over $\Q$. \footnote{Elliptic fibrations on surfaces had already been used in the literature to prove (potential) Zariski density of rational points (see e.g. \cite{Bogomolov2},\cite{articoloV},\cite{RonaldvanLuijk2012}, and \cite{swinnerton-dyer2013}), and conditional results about weak approximation (see e.g. \cite{swinnertonct}, \cite{swinnertondiagonal2}).}

 Hence we feel that the study of the Hilbert Property might shed some light on Question \ref{congetturona}, and, consequently, on the Inverse Galois Problem, perhaps bypassing the much harder problem of the weak weak approximation.

\smallskip

The main results of this paper, Theorem \ref{teoremino} and Theorem \ref{Thm:fibrations}, both provide sufficient conditions for a variety to have the Hilbert Property.


\smallskip

The first theorem provides a sufficient condition for descending the Hilbert Property to a quotient of a variety by a finite group.
	\begin{theorem}\label{teoremino}
    			Let $k$ be a perfect field, and let $X/k$ be a geometrically integral variety. Let $G$ be a finite group acting generically freely on $X$. Assume that there exist Galois field extensions $\left\lbrace L_i/k \right\rbrace _{i \in I}$, with $G_i=\Gal(L_i/k) \cong G$, and isomorphisms $\alpha_i \in \Hom (G_i,G)$ such that:
    	\begin{itemize}
    		\item[(i)] For each $i\in I$, the twist of $X$ by $\alpha_i$ has the Hilbert Property;
    		\item[(ii)] For any finite field extension $E/k$ there exists an $i \in I$ such that $L_i/k$ and $E/k$ are linearly disjoint. 
    	\end{itemize} 
    	Then, the quotient variety $X/G$ has the Hilbert Property.
    \end{theorem}

Observe that in some sense Theorem \ref{teoremino} reverses the connection established by Colliot-Thélène and Sansuc between Question \ref{congetturona} and the Inverse Galois Problem. Indeed, its hypothesis require that the group $G$ is realizable (multiple times) as a Galois group over $k$. This allows us, starting from known positive answers to the Inverse Galois Problem, to exhibit new examples of non-rational unirational varieties with the Hilbert Property. In particular, using Shafarevich's Theorem \cite{shafarevichincorrect}, we deduce the following:
\begin{proposition}\label{Prop:solvable}
	Let $G$ be a finite solvable group acting linearly on an affine space $\A^n/\Q$. Then, the quotient $\A^n/G$ has the Hilbert Property.
\end{proposition}



The second theorem we prove provides a sufficient condition for the Hilbert Property to hold for surfaces with two elliptic fibrations.

\begin{theorem}\label{Thm:fibrations}
	Let $K$ be a number field, and $E$ be a simply connected smooth projective algebraic surface defined over $K$, endowed with two elliptic fibrations $\pi_i:E \rightarrow \P_1/K, \ i=1,2$, such that $\pi_1 \times \pi_2:E \rightarrow \P_1 \times \P_1$ is finite. Suppose that the following hold:
	\begin{itemize}
		\item[(a)] The $K$-rational points $E(K)$ are Zariski-dense in $E$;
		\item[(b)] Let $\eta_1 \cong \Spec K(\lambda)$ be the generic point of the codomain of $\pi_1$. All the diramation points (i.e. the images of the ramification points) of the morphism $\restricts{\pi_2}{\pi_1^{-1}(\eta_1)}$ are non-constant in $\lambda$, and the same holds upon inverting $\pi_1$ and $\pi_2$.
	\end{itemize}
	Then the surface $E/K$ has the Hilbert Property.
\end{theorem}

Theorem \ref{Thm:fibrations} is essentially an elaboration and generalization of the ideas of Corvaja and Zannier \cite{articoloHP}. There are, nevertheless, some key differences in the proof. Namely, Corvaja and Zannier use fibrations which have sections, whereas in our case this is not necessarily so.



\smallskip

The paper is organized as follows. Section 2 contains notations, definitions and preliminaries that will be used throughout the paper. Section 3 is dedicated to the proof of Theorem \ref{teoremino} and some consequences, among which Proposition \ref{Prop:solvable}. Theorem \ref{Thm:fibrations} is proved in Section 4 (which is totally independent from Section 3). Finally, in Section 5 we present some applications of the results of Section 3 and Section 4.
Namely, using a result of Logan, McKinnon and van Luijk \cite{articoloV}, we will apply Theorem \ref{Thm:fibrations} to some diagonal quartic surfaces,
	and we will conclude by proving that a certain quotient of the Fermat surface $F$ by an action of the cyclic group $C_2$ has the Hilbert Property. To prove this last result we employ both Theorems \ref{teoremino} and \ref{Thm:fibrations}.


\begin{acknowledgements}
	The author would like to thank his advisor, Umberto Zannier, for having given him the opportunity to work on these topics, and providing important insights. The author would also like to thank Pietro Corvaja for many fruitful discussions, and Bruno Deschamps, for raising questions that motivated Section 2 of this paper.
\end{acknowledgements}

\section{Background and notation}

This section contains some preliminaries. In particular, in the last paragraph, we shall recall basic facts concerning the Hilbert Property. Moreover, we shall take care here of most of the notation that will be used in the paper. 

\paragraph{Notation}

Throughout this paper, except when otherwise stated, $k$ denotes a perfect field and $K$ a number field. A ($k$-)\textit{variety} is an algebraic quasi-projective variety (defined over the field $k$), not necessarily irreducible or reduced. Unless specified otherwise, we will always work with the Zariski topology. 

Given a morphism $f:X \rightarrow Y$ between $k$-varieties, and a point $s:\Spec(k(s))\rightarrow Y$, we denote by $f^{-1}(s)$ the scheme-theoretic fibered product $\Spec(k(s))\times_Y X$, and call it the \textit{fiber} of $f$ in $s$. Hence, with our notation, this is not necessarily reduced.

\smallskip

A group action of a (finite) group $G$ will always mean a left group action. On the contrary, the Galois action will always be assumed to be a right action. All groups that appear in this paper will be finite. If $\Gamma$ is a group, a $\Gamma$-group is a group $G$ endowed with a homomorphism $\Gamma \rightarrow \Aut (G)$.

\smallskip

When $P$ is a $\bar{k}$-point of an algebraic $k$-variety $V$, and $G_P\subset \Gal (\bar{k}/k)$ is the stabilizer of $P$, we will refer to the field  $L^{G_P}$ as the \textit{field of definition} of $P$.


\smallskip

\paragraph{Twists}

  When $L/k$ is a finite Galois field extension, $G$ is a $\Gal(L/k)$-group, acting on a $k$-variety $X$, and $\alpha = (\alpha_{\sigma}) \in H^1(\Gal(L/k),G)$, we will denote by $X_{\alpha}$ the twisted variety. We refer the reader to \cite[Sec. 2, p.12]{skorobogatov} for the definition and the following fundamental property of twisted varieties.

 \begin{proposition}\label{twist}
 	There exists an isomorphism $\Xi_{\alpha}:X_{\alpha}\times_kL \xrightarrow{\sim} X\times_kL$ such that:
 	\begin{equation}
 	\Xi_{\alpha}(X_{\alpha}(k))
 	=\left\lbrace x \in X(L) \mid x^{\sigma}=\alpha_{\sigma}x \ \forall \sigma \in\Gal(L/k) \right\rbrace ,
 	\end{equation}
 \end{proposition}
 
%
%

\smallskip

\paragraph{Hilbert Property and thin sets}

For a more detailed exposition of the basic theory of the Hilbert Property and thin sets we refer the interested reader to \cite[Ch. 3]{Serre}. We limit ourselves here to recalling some basic definitions and properties.

\begin{definition}\label{Def:thin}
	Let $X$ be a geometrically integral variety, defined over a field $k$. A \textit{thin} subset $S \subset X(k)$ is any set contained in a union $D(k)\cup \bigcup_{i=1,\dots,n} \pi_i(E_i(k))$, where $D \subsetneq X$ is a subvariety, the $E_i$'s are irreducible varieties and $\pi_i:E_i \rightarrow X$ are generically finite morphisms of degree $>1$.
\end{definition}

\begin{remark}\label{HPthin}
	A $k$-variety $X$ has the Hilbert Property if and only if $X(k)$ is not thin.
\end{remark}
  
The following proposition summarizes some basic properties of the Hilbert Property.
\begin{proposition}\label{HP}
	Let $k$ be a perfect field, and $X$ be a geometrically irreducible $k$-variety.
	\begin{itemize}
		\item[(i)] If $X$ has the Hilbert Property and $Y$ is a $k$-variety birational to $X$, then $Y$ has the Hilbert Property. 
		\item[(ii)] If $X$ is a rational variety, and $k$ is a number field, then $X$ has the Hilbert Property.
		\item[(iii)] If $X$ has the Hilbert Property, and $L/k$ is a finite extension, then $X_L$ has the Hilbert Property.
	\end{itemize}
\end{proposition}
\begin{proof}
$(i)$ is an immediate consequence of Remark \ref{HPthin}. It follows from $(i)$ that, in order to prove $(ii)$, it suffices to prove that $\A_n/k$ has the Hilbert Property. This is a consequence of the Hilbert Irreducibility Theorem. We refer the reader to \cite[Ch. 3]{Serre} for the details, and a proof of $(iii)$.
\end{proof}

The following proposition and its corollary, which are implicit in \cite{congetturastrong}, can be found explicitly in \cite[Ch. 3]{Serre}.
\begin{proposition}\label{realizzazioneforte}
	Let $X$ be a geometrically integral variety, defined over a perfect field $k$, and let $G$ be a finite group acting generically freely on $X$. If $X/G$ has the Hilbert Property, then there exist infinitely many linearly disjoint Galois field extensions over $k$ with Galois group $G$.
\end{proposition}

\vskip 0.1mm

\begin{corollary}
	Let $K$ be a number field. Assume a positive answer to Question \ref{congetturona}. Then, all finite groups are realizable as Galois groups over $K$.
\end{corollary}
\begin{proof}
	Let $G$ be a finite group, and let $n \in \N$ be such that there exists a faithful linear action of $G$ on $\A_n/K$. Then, the quotient variety $\A_n/G$ is unirational. Therefore, under our assumptions, it has the Hilbert Property. The statement now follows from Proposition \ref{realizzazioneforte} applied to $X=\A_n$ and $G=G$.
\end{proof}

\section{Descending the Hilbert Property}

In this section we prove Theorem \ref{teoremino} and apply it to quotients of $\A_n$ by linear actions.

\subsection{Proof of Theorem \ref{teoremino}}
%
%
%

In order to prove Theorem \ref{teoremino}, we need the following:
\begin{proposition}\label{corolltipoA}
	Let $k$ be a perfect field, $G$ be a finite group and $L/k$ be a finite Galois field extension with Galois group $G$. Let $X$ be a $k$-variety on which $G$ acts freely, $H < G$ be a subgroup, and $\pi_H:X \rightarrow X/H$ be the quotient map. Let also $\alpha \in H^1(\Gal(L/k),G)=\Hom (\Gal(L/k),G)$ be an isomorphism, $X_{\alpha}$ be the variety $X$ twisted by $\alpha$, and $\Xi_{\alpha}:X_{\alpha}\times_kL \xrightarrow{\sim} X\times_kL$ be defined as in Proposition \ref{twist}.
	Then, the field of definition of each point $P \in \pi_H(\Xi_{\alpha}(X_{\alpha}(k)))$ is $L^H$.
\end{proposition}
\begin{proof}
	Let $\tilde{P} \in X_{\alpha}(k)$, and $P=\pi_H(\Xi_{\alpha}(\tilde{P}))$. Certainly the field of definition of $P$ is contained in $L$, since $\Xi_{\alpha}(X_{\alpha}(k))\subset X(L)$. Hence the field of definition of $P$ is $L^N$, where $N$ is the stabilizer of $P$ through the Galois action of $\Gal(L/k)$. We claim that $N=H$. In fact, if $\sigma \in \Gal(L/k)$, then:
	\[
	P^{\sigma}=(\pi_H(\Xi_{\alpha}(\tilde{P})))^{\sigma}=\pi_H(\Xi_{\alpha}(\tilde{P})^{\sigma})=\pi_H(\alpha_{\sigma}(\Xi_{\alpha}(\tilde{P}))),
	\]
	where $\alpha = (\alpha_{\sigma})$.
	
	Hence we have that $P=P^{\sigma}$ if and only if $\alpha_{\sigma}\Xi_{\alpha}(\tilde{P}) \in H\cdot \Xi_{\alpha}(\tilde{P})$. Since the action of $G$ is free, this proves the proposition.
\end{proof}

\begin{proof}[Proof of Theorem \ref{teoremino}]
	Let $\pi:X \longrightarrow X/G$ be the projection to the quotient. We denote by $k(X)$ and $k(X)^{G}$, respectively, the fields of functions of $X$ and $X/G$. 
	
	
	Suppose, by contradiction, that $X/G$ does not have the Hilbert Property. Then, there exist irreducible covers $\phi_j:  E_j \longrightarrow X/G$, $j \in J$, where $\abs{J}$ is finite and $\deg \phi_j>1$ for each $j$, such that  $X(k) \setminus \bigcup_{j \in J}\phi_j(E_j(k))$ is not Zariski-dense. We can assume, without loss of generality, that the $E_j$'s are geometrically irreducible (see the \textit{Remark on irreducible varieties} in \cite[p. 20]{Serre}).

	Moreover, since the Hilbert Property is a birational invariant (see Proposition \ref{HP}$(i)$), without loss of generality, we may restrict $X$ to any $G$-invariant open subvariety. We will make use of this several times in the proof, and by ``restricting $X$" we will always mean ``restricting $X$ to an open $G$-invariant subvariety".
	
	
	In particular, restricting $X$, we can assume that: 
	\begin{equation}\label{tesidistarobaqua}
	X(k) \subset \bigcup_{j \in J}\phi_j(E_j(k)),
	\end{equation}
	
	
	and the action of $G$ on $X$ is free. 
	
	For $j \in J$, we say that $j$ is \textit{good} if $\bar{k}(X)$ and $\bar{k}(E_j)$ are linearly disjoint over $\bar{k}(X)^G$, and \textit{bad} otherwise. For each $j \in\ J$, we denote by $F_j$ the fibered products $E_j \times_{X/G} X$, and call $\pi_j:F_j \rightarrow X$ the projection on the second factor.

	Note that, if $j$ is good, the tensor product $\bar{k}(X) \otimes_{\bar{k}(X)^{G}} \bar{k}(E_j)$ is a field, and, hence, $F_j$ is geometrically irreducible.
	
	On the other hand, if $j$ is bad, the field extensions $\bar{k}(X)/\bar{k}(X)^{G}$ and $\bar{k}(E_j)/\bar{k}(X)^{G}$ have a common subextension, say $\mathcal{L}_j/\bar{k}(X)^{G}$, with $[\mathcal{L}_j:\bar{k}(X)^{G}]>1$. We call $\iota_{X,j}:\mathcal{L}_j\hookrightarrow \bar{k}(X)$ and $\iota_{E_j}:\mathcal{L}_j\hookrightarrow \bar{k}(E_j)$ the two associated field embeddings. By Galois theory we have that $\iota_{X,j}(\mathcal{L}_j)$ is of the form $\bar{k}(X)^{H_j}$, where $H_j$ is a proper subgroup of $G$. We assume, without loss of generality, that $\iota_{X,j}$ is an inclusion, and therefore $\mathcal{L}_j=\bar{k}(X)^{H_j}$. Therefore, the field embeddings $\iota_{E_j}$ correspond to dominant rational $\bar{k}$-maps $f_{j}:E_j \dashrightarrow X/H_j$. Moreover, restricting $X$, we can assume that, for each bad $j$, the $f_{j}$'s are morphisms (and not just rational maps), defined over $\bar{k}$.
	

	
	Let $k_b\defeq k(\{f_j\}_{\text{bad } j})$ be the minimal common field of definition of all the $f_j$'s, for all the bad $j$'s. Then, the extension $k_b/k$ is finite.
	
	By hypothesis there exists $i \in I$ such that the field $L_i$, as defined in the hypothesis of this theorem, is linearly disjoint with $k_b$. Let $\pi_i$ be the morphism $X_{\alpha_i}\rightarrow X/G$, and $\Xi_{\alpha_i}:X_{\alpha_i} \times_kL_i\rightarrow X \times_kL_i$ the morphism defined in Proposition \ref{twist}. We notice that $\pi_i \times_kL_i= (\pi \circ \Xi_{\alpha_i})\times_kL_i$.
	
	Let us now prove that \[\pi_i(X_{\alpha_i}(k))\nsubseteq \bigcup_{\text{good } j}\phi_j(E_j(k)). \]
	For any good $j$, we denote by $F'_j$ the fibered product $E_j \times_{X/G} X_{\alpha_i}$, and by $\pi'_j:F'_j\rightarrow X_{\alpha_i}$ the projection on the second factor. We note that $F'_j \times_kL\cong F_j \times_kL$, and, hence, the $F'_j$'s are geometrically irreducible and $\pi'_j:F'_j\rightarrow X_{\alpha_i}$ has degree equal to the degree of $\phi_j$, which is $>1$. Then we have that:
	\[
	\pi_i^{-1}\left(\bigcup_{\text{good }j}\phi_j(E_j(k))\right) \cap X_{\alpha_i}(k)=\bigcup_{\text{good }j}\pi_i^{-1}(\phi_j(E_j(k)))\cap X_{\alpha_i}(k)=\bigcup_{\text{good }j}\pi'_j(F'_j(k)).
	\]
	
	Hence, since $X_{\alpha_i}$ has the Hilbert Property, and $\deg \pi'_j>1$ for each good $j$, there exists $Q \in X_{\alpha_i}(k)\setminus \pi_i^{-1}(\bigcup_{\text{good }j}\phi_j(E_j(k)))$. Hence $\pi_i(Q) \in \pi_i(X_{\alpha_i}(k)) \setminus \bigcup_{\text{good }j}\phi_j(E_j(k))\neq \emptyset$, as we wanted to prove.
	
	
%

	Now we claim that: 
	\begin{equation}\label{disgiunto}
	\tag{A}
	\pi_i(X_{\alpha_i}(k)) \cap  \bigcup_{\text{bad }j}\phi_j(E_j(k)) = \emptyset.
	\end{equation}
	
	For any bad $j$, let $\pi_{H_j}:X \rightarrow X/H_j$ and $\xi_j:X/H_j\rightarrow X/G$ be the natural projections.
	We note that the following holds:
	\[\pi_i(X_{\alpha_i}(k)) \cap  \bigcup_{ \text{bad } j}\phi_j(E_j(k))=\bigcup_{\text{bad }j}\xi_j(\pi_{H_j}(\Xi_{\alpha_i}(X_{\alpha_i}(k)))\cap f_j(E_j(k))).\]
	
	Hence, to prove (\ref{disgiunto}), it is enough to prove that: 
	
	\begin{equation}\label{disgiunto1}
	\tag{B}
	\pi_{H_j}(\Xi_{\alpha_i}(X_{\alpha_i}(k)))\cap 
	f_j(E_j(k)) = \emptyset
	\end{equation}
	for each bad $j$.
	
	Let $P \in \pi_{H_j}(\Xi_{\alpha_i}(X_{\alpha_i}(k)))\cap f_j(E_j(k))$ for a bad $j$. Since $j$ is bad and $P$ is contained in $f_j(E_j(k))$, we have that the field of definition of $P$ is contained in $k_b$. On the other hand, we have that $P \in \pi_{H_j}(\Xi_{\alpha_i}(X_{\alpha_i}(k)))$. Hence, by Proposition \ref{corolltipoA}, we have that the field of definition of $P$ is a subextension of $L_i/k$, and, since $H_j \neq G$, $L_i \neq k$. But, by our choice of $i \in I$, $L_i$ and $k_b$ are linearly disjoint, which leads to a contradiction. This concludes the proof of (\ref{disgiunto1}), and, consequently, of (\ref{disgiunto}).
	
	Finally, since we proved that $\pi_i(X_{\alpha_i}(k))\subset X/G(k)$ is not contained in $\bigcup_{\text{good }j}\phi_j(E_j(k))$, we have that $\pi_i(X_{\alpha_i}(k))$ is not contained in $\bigcup_{j \in J}\phi_j(E_j(k))$, which contradicts (\ref{tesidistarobaqua}) and concludes the proof of the theorem.
\end{proof}

\subsection{Quotients by a linear action}


We now exhibit some applications of Theorem \ref{teoremino}.

\begin{definition}\label{strongrealizable}
	We say that a finite group $G$ is \textit{strongly realizable} as a Galois group over $k$ if, for every finite field extension $L/k$, there exists a Galois extension $E/k$, with Galois group $G$, such that $E/k$ and $L/k$ are linearly disjoint.
\end{definition}


\begin{remark}
	We observe that a group $G$ is strongly realizable over $k$ if and only if there are infinitely many Galois field extensions over $k$ with group $G$. Hence, Proposition \ref{realizzazioneforte} states exactly that, when a quotient $X/G$ of a geometrically irreducible variety $X$ by a generically free action of $G$ has the Hilbert Property, then $G$ is strongly realizable.
\end{remark}

We now give a corollary of Theorem \ref{teoremino}. 

\begin{corollary}\label{HPlineare}
	Let $k$ be a perfect field and $G$ be a finite group acting linearly and faithfully on $\A^n/k$ for some $n \geq 1$. Then $\A^n/G$ has the Hilbert Property if and only if $G$ is strongly realizable as a Galois group over $k$.
\end{corollary}
\begin{proof}
	By Proposition \ref{realizzazioneforte}, when $\A^n/G$ has the Hilbert Property, $G$ is strongly realizable.
	
	Conversely, assume that $G$ is strongly realizable.
	Since, by Hilbert's Theorem 90, the twists of $\A^n$ by a linear group action are trivial, the Hilbert Property of $\A^n/G$ follows immediately from Theorem \ref{teoremino}.
\end{proof}

Although Corollary \ref{HPlineare} does not directly help in finding answers to the Inverse Galois Problem, it may be considered of independent interest, since it provides some positive answers to Question \ref{congetturona}
. An example in this direction is Proposition \ref{Prop:solvable}, which we  now prove.


%
%

\begin{proof}[Proof of Proposition \ref{Prop:solvable}]
	Let $G$ be a finite solvable group acting linearly on an affine space $\A_n$. Since quotients of solvable groups are solvable, we may assume without loss of generality that the action of $G$ is faithful. 
	
	The well known theorem of Shafarevich \cite{shafarevichincorrect} states that any finite solvable group is realizable as the Galois group of a Galois extension $L/\Q$, and, as pointed out by Neukirch in \cite[p. 597, Exercise (a)]{Neukirchnovo}, one can actually choose $L$ to have split ramification over any fixed finite set of primes. In particular, $G$ is strongly realizable over $\Q$. Hence, Proposition \ref{Prop:solvable} follows from Corollary \ref{HPlineare}.
\end{proof}

\begin{remark}\label{solvable}
	We observe that, in general, for solvable groups $G$ acting linearly on $\A^n$, the variety $\A^n/G$ may not be (even geometrically) rational. The first such example was given in \cite{Saltman}, where $G$ is a nilpotent group of order $p^9$, and $p$ is (any) prime number. Afterwards, this example has been vastly generalized (see \cite{Bogomolov} and \cite[Sec. 7]{mumbai04}).
\end{remark}

\section{Elliptic fibrations and Hilbert Property}

In this section $k$ will always denote a field of characteristic $0$.

We recall that a finite morphism $f:X \rightarrow Y$ between $k$-varieties is \textit{unramified} (resp. \textit{étale}) in $x \in X$, if its differential $df_x:T_xX\rightarrow T_{f(x)}Y$ is injective (resp. an isomorphism). Otherwise we say that $f$ is \textit{ramified} at $x$. The set of points where $f$ is ramified has a closed subscheme structure in $X$, and we will refer to it as the \textit{ramification locus}. The image of the ramification locus under $f$ is the \textit{diramation locus}. We recall that, by Zariski's Purity Theorem \cite[Lem. 53.20.4, Tag 0BMB]{stacks-project}, when $X$ is normal and $Y$ is smooth, the diramation locus of a finite morphism $f:X \rightarrow Y$ is a divisor. Hence, in this case, we will also refer to the diramation locus as the \textit{diramation divisor}.

\begin{definition}
	A proper morphism $\pi:E \rightarrow C$ between $k$-varieties is an \textit{elliptic fibration} if $C$ is a smooth complete connected $k$-curve and the generic fiber of $\pi$ is a smooth complete geometrically connected curve of genus $1$.
\end{definition}

\begin{definition}\label{Def:doublelliptic}
	A \textit{double elliptic surface} $\Eps=(E, \pi_1,\pi_2)$, defined over a field $k$, is a smooth projective connected algebraic $k$-surface $E$ endowed with two elliptic fibrations $\pi_i:E \rightarrow \P_1/k, \ i=1,2$ such that the map $\pi_1\times\pi_2 : E \rightarrow (\P_1)^2$ is finite.
\end{definition}

\subsection{Proof of Theorem \ref{Thm:fibrations}}

\begin{definition}\label{genericity}
	We say that a double elliptic surface $(E, \pi_1,\pi_2)$, defined over a number field $K$, is \textit{Hilbert generic} 
	if, for any thin subset $T \subset \P_1(K)$, the set
	\[
	\left\{t \in \P_1(K):\#(\pi_1^{-1}(t)\cap\pi_2^{-1}(T))=\infty \right\}
	\]
	is thin, and the same holds upon inverting $\pi_1$ and $\pi_2$.
\end{definition}

We will deduce Theorem \ref{Thm:fibrations} as a consequence of the following proposition.

\begin{proposition}\label{fibrazioni}
	Let $\Eps=(E, \pi_1,\pi_2)$ be a simply connected Hilbert generic\footnote{The hypothesis of Proposition \ref{fibrazioni} can be weakened by requiring that $\Eps$ satisfies the condition of Definition \ref{genericity} just for $\pi_1$ with respect to $\pi_2$ (or viceversa), i.e. for the sake of Proposition \ref{fibrazioni}, one could remove the phrase ``\textit{and the same holds upon inverting $\pi_1$ and $\pi_2$}" from Definition \ref{genericity}. However, since this does not allow to weaken the hypothesis of Theorem \ref{Thm:fibrations}, we decided, for the sake of simplicity, to keep Definition \ref{genericity} as it is presented above.} double elliptic surface, defined over a number field $K$. Suppose that there are two non-thin subsets $N_1,N_2$ of $\P_1(K)$, such that $\#\pi_j^{-1}(x)(K)=\infty$, for $x \in N_j$, $j=1,2$.
	Then, the surface $E/K$ has the Hilbert Property.
\end{proposition}

The proof of Proposition \ref{fibrazioni} uses the following two lemmas. The first one is Lemma 3.2 of \cite{articoloHP}, while an explicitly computable version of the second one can be found in \cite{RonaldvanLuijk2012}. For the sake of completeness, we include here a sketch of its proof.

\begin{lemma}\label{lemmagruppi}
	Let $G$ be a finitely generated abelian group of positive rank. Let $n \in \N$ and $\{h_u+H_u\}_{u=1,\dots,n}$ be a collection of finite index cosets in $G$, i.e. $h_u \in G, \ H_u < G$ and $[G:H_u]<\infty$ for each $u=1,\dots, n$.
	If $G \setminus \bigcup_{u=1,\dots,n} (h_u+H_u)$ is finite, then $\bigcup_{u=1,\dots,n} (h_u+H_u)=G$.
\end{lemma}

\begin{proof}
	See \cite[Lem 3.2]{articoloHP}.
\end{proof}

\begin{lemma}\label{positivegenericrank}
	Let $\pi:\Eps \rightarrow \P_1$ be an elliptic fibration, defined over a number field $K$. Then, there exists an open Zariski subset $U_{\pi} \subset \Eps$ such that, for any $P \in U_{\pi}(K)$, $\pi^{-1}(\pi(P))$ is smooth and $\#{\pi^{-1}(\pi(P))(K)}=\infty$. 
\end{lemma}

\begin{sketch}
	Let $J:\Eps \rightarrow \P_1$ be the Jacobian fibration corresponding to $\pi$, $p_{\lambda}$ be the generic point of $\P_1$, and $J_{\lambda}$ be the generic fiber $J^{-1}(p_{\lambda})$. Let also $H_{\lambda}$ be an ample divisor on the generic fiber $\Eps_{\lambda}\defeq \pi^{-1}(p_{\lambda})$ of $\Eps$,
	\[
	\psi_{\lambda}:\Eps_{\lambda} \rightarrow J_{\lambda}
	\]
	be the morphism defined by $\psi_{\lambda}(Q)\defeq[\deg H \cdot (Q)-H]$, and 
	\[
	\Psi:\pi^{-1}(U)\rightarrow J^{-1}(U)
	\]
	be an extension of $\psi_{\lambda}$ to a Zariski neighbourhood $U$ of $p_{\lambda}$. We can assume, without loss of generality, that $U$ is contained in the good reduction locus of $\pi$. 
	
	Let $N$ be the largest order of torsion of the groups $\tilde{E}(K)$ \footnote{This is a finite number by the Mazur-Merel-Parent Theorem (see e.g. the article by Parent \cite{parent}).}, where $\tilde{E}$ varies in the set of elliptic curves defined over the number field $K$. The claim of the lemma is then satisfied with $U_{\pi} \defeq \pi^{-1}(U)\setminus ([N]\Psi)^{-1}(O)$, where $O$ denotes, with abuse of notation, the image of the zero section of $J$. 
\end{sketch}

\begin{corollary}\label{hypothesisa}
	Let $\Eps=(E,\pi_1,\pi_2)$ be a double elliptic surface, defined over a number field $K$. Then, for each $i=1,2$, there exists a finite set of points $Z_i \subset \P_1(K)$, such that, if $x \in \P_1(K)\setminus Z_i$, for all but finitely many $P \in \pi_{i}^{-1}(x)(K)$, $\#\pi_{3-i}^{-1}(\pi_{3-i}(P))(K)=\infty$.
\end{corollary}
\begin{proof}
	Let
	\[
	Z_i\defeq \{t \in \P_1(K): \pi_i^{-1}(t)\cap (U_{\pi_1}\cap U_{\pi_2})= \emptyset \},
	\]
	where $U_{\pi_i}\subset E$ is the open Zariski subset defined in Lemma \ref{positivegenericrank}.
	The statement now follows from Lemma \ref{positivegenericrank}.
\end{proof}

\begin{proof}[Proof of Proposition \ref{fibrazioni}]

	Suppose, by contradiction, that the surface $E$ does not have the Hilbert Property. Then there exists a finite collection of covers $\phi_i:Y_i \rightarrow E$, $i\in I$, with $\deg \phi_i \defeq d_i>1$, such that $E(K) \subset \cup_i \phi_i(Y_i(K))\cup D(K)$, where $D$ is a proper closed subvariety of $E$. Without loss of generality, we may assume that the $Y_i$ are geometrically irreducible\footnote{See the \textit{Remark on irreducible varieties} in \cite[p. 20]{Serre}} and normal, and that the $\phi_i$ are finite maps \footnote{In fact, the last two assumptions are true up to a birational morphism of $Y_i$, hence, by enlarging the subvariety $D$, we may assume that they hold.}.
	Since the surface $E$ is algebraically simply connected, each of the $\phi_i$ has a nontrivial diramation divisor. Call $R_i \subset E$ the diramation divisor of $\phi_i$, and let \[E^{j,x} \defeq \pi_j^{-1}(x), \quad Y^{j,x}_i\defeq (\pi_j\circ \phi_i)^{-1}(x).\]

	Since the proof is rather technical, we will first describe the general idea. In the simplest possible case we would have that, for \textit{all} $i$, the generic geometric fiber of $\pi_1 \circ \phi_i$ is irreducible, \textit{and} $R_i$ has at least one irreducible component that is transverse to the fibers of $\pi_1$ \footnote{Most of the hypothesis of the proposition are not needed under these assumptions. Namely, to complete the proof in this case, one just needs one elliptic fibration, which we chose without loss of generality to be $\pi_1$. Moreover, it suffices that $N_1$ is infinite, not necessarily non-thin.}. Indeed, in this case, let $t_0 \in N_1$ be a  (sufficiently Zariski-generic) $K$-rational point. Without loss of generality, $E^{1,t_0}\defeq \pi_1^{-1}(t_0)$ will then be a smooth curve of genus $1$ with infinitely many $K$-rational points, $E^{1,t_0} \cap D$ will be finite, and, for \textit{all} $i$, the curve $\phi_i^{-1}(E^{1,t_0})=Y_i^{1,t_0}$ will be smooth and (geometrically) irreducible, and the morphism $Y_i^{1,t_0}\rightarrow E^{1,t_0}$ will be ramified over $R_i \cap E^{1,t_0}$. Hence, by Riemann-Hurwitz' Theorem, for \textit{all} $i$, $Y_i^{1,t_0}$ has genus $>1$, whence, by Faltings' Theorem, $Y_i^{1,t_0}$ has finitely many $K$-rational points. It follows that $(\cup_i \phi_i(Y_i(K))\cup D(K))\cap E^{1,t_0}(K)$ is a finite set, which leads to the desired contradiction.
	
	In the general case one needs to take care of two problems: there will be $i \in I$ (which are later going to be called $1$-\textit{bad}) such that the generic geometric fiber of $\pi_1 \circ \phi_i$ is \textit{not} irreducible, and there will be $i \in I$ (which are later going to be called $1$-\textit{almost good}) such that $R_i$ is contained in finitely many fibers of $\pi_1$. We solve the first problem by choosing $t_0$ outside of a specific thin subset of $\P_1(K)$ (which will later be called $S'_1$). This guarantees that for bad $i$'s $\phi_i^{-1}(E^{1,t_0})(K)=\emptyset$. To solve the second problem we notice that, in this case, $R_i$ is transverse to the fibers of $\pi_2$, which allows to use an argument \textit{similar} to the one employed in the simplest case, but taking into account both fibrations\footnote{One cannot simply reduce to the ``simplest" case on the fibration $\pi_2$, unless $\abs{I}=1$, i.e. there is only one cover. The main difficulty is that it might happen that some of the $R_i$ are contained in the fibers of $\pi_1$ and others are contained in the fibers of $\pi_2$. Lemma \ref{lemmagruppi} will be the key ingredient that will allow to overcome these difficulties.}.

	Here are the details of the proof.
	
	For each $j=1,2$, we divide the covers into three types through the following partition $I=I_j^{g}\cup I_j^{ag}\cup I_j^b$:
	\begin{description}
		\item[]   $i\in I_{j}^{ag}$ if $\pi_j \circ \phi_i$ has a geometrically irreducible generic fiber, and $R_i$ is contained in a finite number of fibers of $\pi_j$; in this case we  say that $i$ and the corresponding  cover $\phi_i:Y_i\rightarrow E$ are {\em almost $j$-good};
		\item[] $i\in I_{j}^g$ if $\pi_j \circ \phi_i$ has a geometrically irreducible generic fiber, and the morphism $\pi_j\arrowvert_{R_i}:R_i \rightarrow \P_1$ is surjective; in this case we  say that $i$ and the corresponding cover $\phi_i:Y_i\rightarrow E$ are {\em $j$-good};
		\item[] $i\in I_{j}^b$  if the morphism $\pi_j \circ \phi_i$ has a geometrically reducible generic fiber; in this case we  say that $i$ and the corresponding  cover $\phi_i:Y_i \rightarrow E$ are {\em $j$-bad}.
	\end{description}

	
	When $\phi_i$ is $j$-bad, we consider the relative normalization decomposition (see e.g. \cite[Def. 28.50.3, Tag 0BAK]{stacks-project} or \cite[Def. 4.1.24]{Liu}) of the morphism $\pi_j \circ \phi_i$:
	\begin{equation}\label{fattorizzazione}
	\pi_j \circ \phi_i:Y_i \xrightarrow{\phi_{i,j}} C_{i,j} \xrightarrow{r_{i,j}} \P_1,
	\end{equation}
	where the $C_{i,j}$'s are normal projective $K$-curves, and the $r_{i,j}$'s are finite morphisms, with $\deg r_{i,j} >1$ for each $i \in I$, $j=1,2$. Moreover, since the $Y_i$'s are geometrically irreducible, so are the curves $C_{i,j}$'s.
	
	
	Now, for each $j=1,2$, we define the two following subsets $S'_j,S''_j \subset \P_1(K)$:
	\begin{flalign*}
		S'_j \defeq &\cup_{i \in I_j^b} C_{i,j}(K)\subset \P_1(K)\\
		S''_j\defeq &\{x \in \P_1(K):\ E^{j,x}\subset D  \cup \cup_{i \in I_j^{ag}} R_i\}\cup \\
		 &\cup_{i \in I_j^g \cup I_j^{ag}}\{x \in \P_1(K):\ Y_i^{j,x} \text{ is not irreducible and smooth} \} \cup\\
		 & \cup \{x \in \P_1(K):\ E^{j,x} \text{ is not irreducible and smooth} \}.
	\end{flalign*}
	
%
	Observe that for each $j=1,2$, $S'_j$ is thin, while, by the Theorems of generic smoothness (see e.g. \cite[Cor. 10.7]{hartshorne}) and generic irreducibility (see e.g. \cite[Lem. 36.25.5, Tag 0553]{stacks-project}), $S''_j$ is finite. Let now $Z_j\subset \P_1(K), \ j=1,2$ be defined as in Corollary \ref{hypothesisa}, and let $T_j\defeq S'_j \cup S''_j\cup Z_j$
	. We notice that $T_j$ is thin. 
	
	To complete the proof of Proposition \ref{fibrazioni} we need the following two lemmas.
	
	\begin{lemma}\label{Lem1}
If $x \in \P_1(K)\setminus T_j$, then $\#\left(E^{j,x}(K)\cap (\cup_{i \in I_j^g} \phi_i(Y_i(K)))\right)<\infty$.
	\end{lemma}
	
	\begin{proof}
Since $x \notin S''_j$, we have that $Y_i^{j,x}$ is a smooth irreducible curve, and $R_i$ intersects properly $E^{j,x}$. It follows that the morphism 
\[
{\phi_i\arrowvert_{Y_i^{j,x}}}:Y_i^{j,x} \rightarrow E^{j,x}
\]
 is ramified by the invariance of the ramification locus under generic base change (see e.g. \cite[Tag 0C3H]{stacks-project}).
Hence, by the Riemann-Hurwitz formula applied to $\phi_i\arrowvert_{Y_i^{j,x}}$, the smooth curve $Y_i^{j,x}$ has genus $>1$. Therefore, by Faltings' theorem, $Y_i^{j,x}(K)$ is finite. Lemma \ref{Lem1} now follows from the following equality:
\begin{equation*}
	E^{j,x}(K)\cap \left(\bigcup_{i \in I_j^g} \phi_i(Y_i(K))\right)=\bigcup_{i \in I_j^g} \phi_i(Y_i^{j,x}(K)).
\end{equation*}
\end{proof}

	\begin{lemma}\label{Lem2}
		If $x \in \P_1(K)\setminus T_j$ and $\#E^{j,x}(K)=\infty$, then $E^{j,x}(K)\subset \bigcup_{i \in I_j^{ag}} \phi_i(Y_i(K))$.
	\end{lemma}
	
	\begin{proof}
		
		We have that $E^{j,x}(K) \subset \cup_{i \in I} \phi_i(Y_i(K))\cup D(K)$. Since $x \notin S'_j$, by the decomposition (\ref{fattorizzazione}) we have that $E^{j,x}(K)\cap (\cup_{j-\text{bad }i} \phi_i(Y_i(K)))=\emptyset$. Hence:

		\begin{equation}\label{Eq:star3} E^{j,x}(K) \subset (D\cap E^{j,x})(K) \cup \bigcup_{i \in I_j^g} \phi_i(Y_i^{j,x}(K)) \cup \bigcup_{i \in I_j^{ag}} \phi_i(Y_i^{j,x}(K)).
		\end{equation}
		
		Since $x \notin S''_j$, we have that $(D\cap E^{j,x})(K)$ is finite. Moreover, we know by Lemma \ref{Lem1} that $\cup_{i \in I_j^g} \phi_i(Y_i^{j,x}(K))$ is finite. Therefore, we immediately deduce from (\ref{Eq:star3}) that:
		
		
		
		\begin{equation}\label{tuttomenofinito}
		E^{j,x}(K) \subset \bigcup_{i \in I_j^{ag}} \phi_i(Y_i^{j,x}(K))\cup A_0,
		\end{equation}
		where $A_0$ is a finite set. 
		
		We notice now that, for $i \in I_j^{ag}$, $Y_i^{j,x}$ is a smooth complete curve of genus $1$ (in fact, since $R_i \cap E^{j,x}=\emptyset$, the morphism $\phi_i\arrowvert_{Y_i^{j,x}}:Y_i^{j,x} \rightarrow E^{j,x}$ is unramified).
		Now $\phi_i\arrowvert_{Y_i^{j,x}}:Y_i^{j,x} \rightarrow E^{j,x}$ is a morphism between genus $1$ curves. Hence, if $Y_i^{j,x}(K)\neq \emptyset$, it is the composite of a translation and an isogeny (after a choice of origin for $Y_i^{j,x}$ and $E^{j,x}$). Therefore, by the weak Mordell-Weil Theorem, $\phi_i(Y_i^{j,x}(K)) \subset E^{j,x}(K)$ is either empty or a finite index group coset.
		
		Hence $\cup_{i \in I_j^{ag}}\phi_i(Y_i^{j,x}(K))\subset E^{j,x}(K)$ is a union of finite index group cosets, and Lemma \ref{Lem2} follows from (\ref{tuttomenofinito}) and Lemma \ref{lemmagruppi}.
	\end{proof}
	
	We may now conclude the proof of Proposition \ref{fibrazioni}.

	Let
	\[
	\widetilde{T_1} \defeq 	\left\{t \in \P_1(K):\#(\pi_2^{-1}(t)(K)\cap\pi_1^{-1}(T_1))=\infty \right\}.
	\]
	Since $\Eps$ is Hilbert generic, we have that $\widetilde{T_1} \subset \P_1(K)$ is thin. 
	Let now $x \in N_2 \setminus (T_2 \cup \widetilde{T_1})$, and
	\[
	X \defeq \{Q \in E^{2,x}(K) : \pi_1(Q) \in T_1 \text{ or } \#E^{1,\pi_1(Q)}(K)<\infty\}.
	\]
	
	Since $x\notin \widetilde{T_1}$ and $x \notin Z_2$, we have that $\# X<\infty$, and, since $x \in N_2$, $\#E^{2,x}(K)=\infty$. Now, for any $Q \in E^{2,x}(K) \setminus X$, we have that $\pi_1(Q) \notin T_1 \text{ and } \#E^{1,\pi_1(Q)}=\infty$, and, therefore, by Lemma \ref{Lem2}, $Q \in \cup_{i \in I_1^{ag}} \phi_i(Y_i(K))$. It follows that $E^{2,x}(K)\setminus X \subset \cup_{i \in I_1^{ag}} \phi_i(Y_i(K))$.
 	
 	Since $\pi_1\times \pi_2$ is finite, no irreducible component of the fibers of $\pi_1$ is contained in a fiber of $\pi_2$. Therefore almost $1$-good covers are $2$-good covers, and, hence, by Lemma \ref{Lem1}, $E^{2,x}(K) \cap \cup_{i \in I_1^{ag}} \phi_i(Y_i(K))$ is finite. Since $E^{2,x}(K) \setminus X$ is infinite, we obtain a contradiction.
\end{proof}

\begin{lemma}\label{coverellittiche}
	Let $K$ be a number field, and let $\pi_i:E_{i}\longrightarrow \P_1/K, \ i\in I$, with $\abs{I}=\infty$,  be elliptic covers (i.e. finite morphisms where $E_i$ is a smooth geometrically connected projective genus $1$ curve), such that: 
	\begin{itemize}
		\item[(a)] For each choice $p_1, \dots, p_n$ of points in $\P_1$, we have that for all but finitely many $i\in I$, $E_i \longrightarrow \P_1/K$ does not ramify over the $p_i$'s;
		\item[(b)] For each $i \in I$ there is a subset $S_i \subset E_i(K)$ of infinite cardinality.
	\end{itemize}
	Then $\bigcup_i \pi_i(S_i)$ is not thin in $\P_1(K)$. 
\end{lemma}
\begin{proof}
	Suppose, by contradiction, that $\bigcup_i \pi_i(S_i)$ is thin. 
	Then there exist a finite set $A \subset \P_1(K)$, a $m \in \N$, and, for each $j=1,\dots,m$,  a smooth complete curve $C_j$ defined over $K$, and finite morphisms $\phi_j: C_j \rightarrow \P_1$, of degree $>1$, such that:
	\[
	\bigcup_i \pi_i(S_i)\subset \bigcup_{j=1,\dots,m}\phi_i(C_i(K))\cup A.
	\]

	
	Let $\{p_1, \dots, p_n\} \in \P_1(\bar{K})$ be the union of all diramation points of the $\phi_j$'s. By hypothesis $(a)$, there exists $i_0 \in I$ such that $\pi_{i_0}: E_{i_0} \longrightarrow \P_1$ is an elliptic cover with ramification disjoint from the $p_k$'s. Let $D_j\defeq E_{i_0}\times_{\P_1}C_j$, and $\psi_j\defeq D_j \longrightarrow E_{i_0}$ be the projection on the first factor. Since the diramations of $\phi_j$ and $\pi_{i_0}$ are disjoint and $\P_1$ is simply connected, $D_j$ is a smooth irreducible curve. Moreover, since $\phi_j$ ramifies over points where $\pi_{i_0}$ does not, the morphism $\psi_j$ will be ramified over the ramification points of $\phi_j$. Therefore, by applying the Riemann-Hurwitz formula on $\psi_j$, we deduce that the genus of $D_j$ is $\geq 2$. Hence, by Faltings' theorem, there are only finitely many $K$-rational points on $D_j$. It follows that $\pi_{i_0}^{-1}(\phi_j(C_j(K)))=\psi_j(D_j(K))$ is finite for each $j=1,\dots,m$, which implies that there are infinitely many points $p$ in $S_{i_0}$ such that $\pi_{i_0}(p)$ does not lie in any of the $\phi_j(C_j(K))$, which is a contradiction.
\end{proof}

We are now ready to prove Theorem \ref{Thm:fibrations}.

\begin{proof}[Proof of Theorem \ref{Thm:fibrations}]
	
	Theorem \ref{Thm:fibrations} follows directly from Proposition \ref{fibrazioni} once we have shown that the hypothesis of the proposition are satisfied, i.e. that the double elliptic surface $\Eps \defeq (E,\pi_1,\pi_2)$ is Hilbert generic, and that there exist two non-thin subsets $N_1,N_2 \subset \P_1(K)$, such that, for each $x \in N_j, j=1,2, \#\pi_j^{-1}(x)(K)=\infty$.
	
	\smallskip
	
	We first prove that the double elliptic surface $\Eps/K$ is Hilbert generic. 
	
	For any $T \subset \P_1(K)$ let
	\[
	A_T\defeq\left\{t \in \P_1(K):\#(\pi_1^{-1}(t)\cap\pi_2^{-1}(T))=\infty \right\}.
	\]
	Since there is complete symmetry between $\pi_1$ and $\pi_2$, to prove that $\Eps/K$ is Hilbert generic it will suffice to prove that $A_T$ is finite (hence thin) for any thin subset $T \subset \P_1(K)$. 
	
	Suppose by contradiction that, for some thin subset $T \subset \P_1(K)$, $A_T$ is infinite. Hypothesis $(b)$ grants that hypothesis $(a)$ of Lemma \ref{coverellittiche} holds for the morphisms $\phi_t\defeq \restricts{\pi_2}{\pi_1^{-1}(t)}$, indexed by $t \in A_T$. Hence, by applying Lemma \ref{coverellittiche} to the morphisms $\phi_t\defeq \restricts{\pi_2}{\pi_1^{-1}(t)}$, indexed by $t \in A_T$, and to the sets $S_t \defeq \pi_1^{-1}(t)\cap\pi_2^{-1}(T)$, we deduce that $T \supset\cup_{t \in A_T} \pi_2(\pi_1^{-1}(t)\cap\pi_2^{-1}(T))$ is not thin, which is a contradiction. Therefore, $\Eps/K$ is Hilbert generic.

	
	Let us now prove that there exists a non-thin subset $N_1 \subset \P_1(K)$ such that, for each $x \in N_1$,  $\#\pi_1^{-1}(x)(K)=\infty$. The argument proving the existence of $N_2 \subset \P_1(K)$ is symmetric upon exchanging $\pi_1$ and $\pi_2$. 
	
	Let $Z_j\subset \P_1(K), \ j=1,2$, be defined as in Corollary \ref{hypothesisa}. Lemma \ref{positivegenericrank} and the fact that $K$-rational points are Zariski-dense in $E$ (hypothesis $(a)$) imply that there exists a $t \in \P_1(K)\setminus Z_1$ such that $\#\pi_1^{-1}(t)(K)=\infty$. By Corollary \ref{hypothesisa}, there exists an infinite subset $X \subset \pi_1^{-1}(t)(K)$ such that for each $P \in X$
	there exists an infinite subset $X_P \subset \pi_2^{-1}(\pi_2(P))(K)$ such that for each $Q \in X_P$, $\#\pi_1^{-1}(Q)(K)=\infty$. 
	
%
	
	By Lemma \ref{coverellittiche}, applied to the morphisms $(\restricts{\pi_1}{\pi_2^{-1}(\pi_2(P))})_{P \in X}$ and to the sets $(X_P)_{P \in X}$, we know that $\cup_{P \in X} \pi_1(X_P) \subset \P_1(K)$ is non-thin. Then the set $N_1\defeq \cup_{P \in X} \pi_1(X_P) \subset \P_1(K)$ satisfies the sought condition. This concludes the proof of the theorem.
\end{proof}
%

\section{K3 surfaces with the Hilbert Property}

In this section we apply Theorem \ref{teoremino} and Theorem \ref{Thm:fibrations} to prove that some families of K3 surfaces have the Hilbert Property.

\subsection{Diagonal quartic surfaces}\label{diagonalquarticsurfaces}

\begin{definition}\label{Def:diagonal}
	For a field $k$ and $a,b,c,d \in k^*$, let $V_{a,b,c,d}$ be the surface in $\P_3/k$ defined by the equation $ax^4+by^4+cz^4+dw^4=0$. The surfaces $V_{a,b,c,d} \subset \P_3$ are called \textit{diagonal quartic surfaces}.
\end{definition}

Diagonal quartic surfaces have been widely studied in the literature. In particular, in \cite{articoloV}, Logan, McKinnon and van Luijk prove the following theorem concerning the Zariski-density of rational points in certain diagonal quartic surfaces.


\begin{theorem}[Logan, McKinnon, van Luijk]\label{articoloV}
	Let $V_{a,b,c,d}:ax^4+by^4+cz^4+dw^4=0$ be a diagonal quartic surface in $\P_3/\Q$, and suppose that $abcd \in (\Q^*)^2$. Suppose, moreover, that there exists a rational point $P=[x_0:y_0:z_0:w_0] \in V_{a,b,c,d}(\Q)$ such that $x_0y_0z_0w_0 \neq 0$ and $P$ does not lie on the $48$ lines contained in $V_{a,b,c,d}$.
	Then, the rational points $V_{a,b,c,d}(\Q)$ are Zariski-dense in $V$. 
\end{theorem}

The 48 lines lying on the surface $V_{a,b,c,d}$ may be described explicitly through the following equations:
\begin{align*}
&\begin{cases}
\sqrt[4]{a}x=i^{j/8}\sqrt[4]{b}y \ j=1,3,5,7\\
\sqrt[4]{c}z=i^{k/8}\sqrt[4]{d}w \ k=1,3,5,7
\end{cases},
\end{align*}
and their permutations (permuting the variables $x,y,z,w$ and the coefficients $a,b,c,d$ by the same element in $\mathcal{S}_4$).

We denote by $\Omega_{a,b,c,d}$ the ``bad" subvariety of Theorem \ref{articoloV}, i.e.
\[
\Omega_{a,b,c,d}\defeq \{[x:y:z:w]\in \P_3:xyzw=0 \}\cup L_{a,b,c,d},
\]
where $L_{a,b,c,d}$ denotes the union of the 48 lines in $V_{a,b,c,d}$.

As a corollary of Theorem \ref{articoloV} (see Corollary \ref{Pr:diagonalquartics} below), using Theorem \ref{Thm:fibrations}, we will deduce that these surfaces have the Hilbert Property.

In \cite{articoloV}, the authors use two elliptic fibrations $\pi_1,\pi_2:V_{a,b,c,d}\rightarrow \P_1$ of the surface $V_{a,b,c,d}$ to prove Theorem \ref{articoloV}. Since we are going need these fibrations to prove Corollary \ref{Pr:diagonalquartics}, we briefly recall their construction, while we refer the reader to \cite[Sec. 2]{articoloV} for more details.

To shorten the notation, let $V\defeq V_{a,b,c,d}$, $P$ be a point in $V(\Q)$, and $Q \subset \P_3/\Q$ be the quadric:
\[
ax^2+by^2+cz^2+dw^2=0.
\] 
Let $S:V\rightarrow Q$ be defined as $S([x:y:z:w])=[x^2:y^2:z^2:w^2]$, and $P'=S(P)\in Q(\Q)$. Now $Q \subset \P_3$ is a quadric with a prescribed rational point $P' \in Q(\Q)$, and with $abcd \in (\Q^*)^2$. Hence, the two line pencils of $Q$ are defined over $\Q$ and they define an isomorphism $(\psi_1,\psi_2):Q \xrightarrow{\cong} \P_1\times \P_1$, which is unique up to a permutation of the two pencils and linear automorphisms of the $\P_1$ parameterizing them. We fix once and for all one such isomorphism, which we denote by $(\psi_1,\psi_2)$.
 Define \[(\pi_1,\pi_2)\defeq (\psi_1,\psi_2)\circ S:V \rightarrow \P_1 \times \P_1.\]
The two morphisms $\pi_1,\pi_2:V_{a,b,c,d} \rightarrow \P_1$ are then elliptic fibrations, and the morphism $(\pi_1,\pi_2):V_{a,b,c,d} \rightarrow \P_1 \times \P_1$ is finite.
%

\begin{corollary}\label{Pr:diagonalquartics}
	Let $V_{a,b,c,d}:ax^4+by^4+cz^4+dw^4=0$ be a diagonal quartic surface in $\P_3/\Q$, and suppose that $abcd \in (\Q^*)^2$. Suppose moreover that there exists a rational point $P=[x_0:y_0:z_0:w_0] \in V_{a,b,c,d}(\Q)\setminus \Omega_{a,b,c,d}$.
Then, the surface $V_{a,b,c,d}/\Q$ has the Hilbert Property. 
\end{corollary}
\begin{proof}
	Let $\mathcal{V}\defeq (V_{a,b,c,d},\pi_1,\pi_2)$, where we are keeping the notation of the above discussion. 
	
	We will show that $\mathcal{V}$ satisfies the two hypothesis of Theorem \ref{Thm:fibrations}. 
	
	The Zariski density of rational points of $V_{a,b,c,d}$ follows directly from Theorem \ref{articoloV}, so hypothesis $(a)$ is satisfied. 
	
	In order to verify hypothesis $(b)$, let $Q^{\lambda}/\Spec \Q(\lambda)$ and  $V^{\lambda}/\Spec \Q(\lambda)$  be the generic fibers, respectively, of $\psi_1$ and $\pi_1$.  From the definition of the map $S$, it follows directly that the morphism $S^{\lambda}\defeq \restricts{S}{V^{\lambda}}:V^{\lambda}\rightarrow Q^{\lambda}$ ramifies over the locus $xyzw=0$. It is immediate that $\restricts{\psi_2}{Q^{\lambda}}$ is an isomorphism, hence it is unramified. Therefore the ramification of $\restricts{\pi_2}{V^{\lambda}}=\restricts{\psi_2}{Q^{\lambda}}\circ S^{\lambda}$ is contained in the locus $\{xyzw=0\}\subset V^{\lambda}$. However, the locus $\{xyzw=0\}\subset V$ intersects properly the fibers of $\pi_2$. Hence the diramation of $\restricts{\pi_2}{V^{\lambda}}=\restricts{\psi_2}{Q^{\lambda}}\circ S^{\lambda}$ is non-constant in $\lambda$. 
	
	After reiterating the above argument upon exchanging $\pi_1$ and $\pi_2$, we have proven that hypothesis $(b)$ of Theorem \ref{Thm:fibrations} holds for $\mathcal{V}$, thus concluding the proof of the corollary.
\end{proof}

\begin{remark}\label{Fermat}
	The case $(a,b,c,d)=(1,1,-1,-1)$ is that of the Fermat surface, whose Hilbert Property has been proven in \cite{articoloHP}.
\end{remark}

\subsection{An application of Theorem \ref{teoremino}}

Let $\tilde{Y}$ be the smooth minimal projective surface over $\Q$, which is birational to the following surface of $\A^4$: 


\begin{equation}
	Y^0\defeq \begin{cases}
		x^4+y^2=z^2+w^4\\
		yz=1
	\end{cases}.
\end{equation}


We will show (see Proposition \ref{Prop:quotient}) that the surface $Y^0$ (and, hence, its smooth model $\tilde{Y}$, which is a K3 surface) has the Hilbert Property. Rather than the result itself we feel that the interest of this example resides in the fact that it is a non-trivial application of Theorem \ref{teoremino}. In fact, the surface $Y^0$ is (isomorphic to) an open Zariski subvariety of the quotient\footnote{$Y$ is isomorphic to the projective closure of $Y^0$ in the weighted projective space $\P(12212)/\Q$, with projective variables $x,y,z,w,t$ weighted in this order. The surface $Y$ is singular in the 4 points $[1:0:0:i^k:0], \ k=0,\dots,3$. Blowing up these 4 points yields a smooth K3 surface $\tilde{Y}$.} $Y$ of the Fermat quartic $F$ by the automorphism $\sigma([x:y:z:w])=[x:-y:-z:w]$.
 \smallskip

For any $d \in \Q^*$, let $F_d \defeq V_{1,d^2,-d^2,-1}$, and $\Omega_d\defeq \Omega_{1,d^2,-d^2,-1}$.

\begin{lemma}\label{Cond:final}
	 For any $l \in \N$, and $a_1,\dots,a_l \in \Q^*$, there exists $d \in \Q^*$ such that $F_d(\Q) \nsubseteq \Omega_d$ and, for all $j =1,\dots,l \ , \ da_j \notin (\Q^*)^2\subset \Q^*$.
\end{lemma}
\begin{proof}
	In order to prove Lemma \ref{Cond:final}, we first construct rational numbers $d_k \in \Q^*, \ k \in \N$, and rational points $P_k \in F_{d_k}(\Q)$. The lemma will then follow after we prove that, for any $l \in \N$, and $a_1,\dots,a_l \in \Q^*$, there exists $k \in \N$ such that $F_{d_k}(\Q) \nsubseteq \Omega_{d_k}$ and, for all $j =1,\dots,l \ , \ d_ka_j \notin (\Q^*)^2$. 
	
	We now describe in detail this construction, which concludes with the definition of the points $P_k$ in (\ref{points}).

	
	
	We choose a pair of odd coprime integers $n,m>1$, and let $\lambda=n^4-m^4$. Let $C_{\lambda}$ be the normalization of the projective closure of the following curve in $\A_2$ (with coordinates $d$ and $t$):
	\begin{equation}\label{twistata}
		C_{\lambda}^0:\lambda d^2=t^4-1,
	\end{equation}
	The curve $C_{\lambda}$ is a genus $1$ curve defined over $\Q$. By construction, the point $P_0=(d_{P_0},t_{P_0})=\left(\frac{1}{m^2},\frac{n}{m}\right)$ lies on $C_{\lambda}$. The curve $C_{\lambda}$ is a quadratic twist of the (normalization of the projective closure of the) curve:
	\begin{equation}\label{Quartica}
	C:d^2=t^4-1,
	\end{equation}
	with respect to the involution $\sigma_C: d \mapsto -d, \ t \mapsto t$, and quadratic field $\Q(\sqrt{\lambda})$. 
	Let $\Xi_{\lambda}:C_{\lambda}\times \Spec \Q(\sqrt{\lambda}) \rightarrow C \times \Spec \Q(\sqrt{\lambda})$ be the morphism defined in Proposition \ref{twist}, for $X=C$ and $\alpha$ the unique isomorphism between $\Gal(\Q(\sqrt{\lambda})/\Q)$ and $<\sigma_C>\subset \Aut (C)$. It is easy to see that $\Xi_{\lambda}(P_0)=\left(\frac{\sqrt{\lambda}}{m^2},\frac{n}{m}\right)$.
	
	
	Using the substitution $x=2(t^2-d)$, $y=2tx$ one obtains the following Weierstrass model for $C$:
	\begin{equation}\label{Weierstrass}
	y^2=x^3+4x,
	\end{equation}
	which naturally endows $C$ with a structure of elliptic curve (and, therefore, a choice of origin $O \in C(\Q)$). In what follows, we are going to use tacitly such structure.
	
	It is a straightforward verification to check that $\sigma_C(P)=O_2-P$, where $O_2 \in C(\Q)$ is the $2$-torsion point with (Weierstrass) coordinates $(x(O_2),y(O_2))=(0,0)$. Hence:
	\[
	\Xi_{\lambda}(C_{\lambda}(\Q))=\{P \in C(\Q(\sqrt{\lambda})):P+\bar{P}=O_2\},
	\]
	and odd multiples of $P_0$ on $C$ will correspond to rational points on $C_{\lambda}$. Let us now prove that $P_0$ is not a torsion point in $C$.

	By explicit computation, that we omit here, it can be seen that $x([2]P_0)=2(t(P_0)^2-1/t(P_0)^2)$, where $t(P_0)$ denotes the $t$-coordinate of $\Xi_{\lambda}(P_0)$ in the model (\ref{Quartica}), and $x([2]P_0)$ denotes the $x$-coordinate of $[2]P_0$ in the model (\ref{Weierstrass}). Since $t(P_0)^2-1/t(P_0)^2=\frac{2(n^4-m^4)}{n^2m^2}\notin \Q$, it follows from the classical Theorem of Lutz and Nagell (see e.g. \cite[Cor. VIII.7.2]{silverman1}), that $[2]P_0$, and hence $P_0$, is not torsion in $C$.

	
	
	For each $k \in \N$, let now $(d_k,t_k)$ be the coordinates in the model (\ref{twistata}) of the point $\Xi_{\lambda}^{-1}([2k+1]P_0) \in C_{\lambda}(\Q)$. Such coordinates give us infinitely many rational solutions $(d_k,t_k)_{k \in \N} \in \Q^2$ of the polynomial equation
	\[
	\lambda d^2=t^4-1.
	\]	
	Moreover, for each $k\in \N$, we have that the point
	\begin{equation}\label{points}
	P_k\defeq [t_k:m:n:1]
	\end{equation}
	lies in $F_{d_k}(\Q)$. By a straightforward calculation, that we omit here, one can verify that, for all but finitely many $k \in \N$, $P_k \notin \Omega_{d_k}$. Hence, for all but finitely many $k \in \N$, $F_{d_k}$ has the Hilbert Property.
	
	

	
	Now we prove that for any $l\in \N$, and for any $a_1,\dots,a_l \in \Q^*$, there are only finitely many $k \in \N$ such that, for some $j$, $d_ka_j \in (\Q^*)^2$.
	
	We have that, if $d_ka_j=s^2$ for some $s \in \Q$, then:
	\begin{equation}\label{Fermat_twistata}
	a_j^{-2}\lambda s^4=t^4-1\  (\text{where } a_j \neq 0).
	\end{equation}
	Equation (\ref{Fermat_twistata}) defines an affine curve $D^0_j$ in the affine plane $\A_2$ with coordinates $s$ and $t$. Its projective closure $D_j\subset \P_2$ is a curve of (geometric) genus $3$ (in fact $D_j$ is a smooth curve of degree $4$ in $\P_2$, and has therefore genus $(4-1) \cdot (4-2)/2=3$). 
	
	Therefore, by Faltings' Theorem applied to the curves $D_j$, for each $j=1,\dots, l$, there are only finitely many rational solutions $(s,t)\in \Q^2$ of the Equation (\ref{Fermat_twistata}). Hence, for each $j=1,\dots,l$ there are only finitely many $k \in \N$ such that $a_jd_k$ is a square, which concludes the proof of Lemma \ref{Cond:final}.
\end{proof}

\begin{proposition}\label{Prop:quotient}
	Let $F:\{x^4+y^4=z^4+w^4\}\subset \P_3/\Q$ be the Fermat quartic, and let $\sigma$ be the automorphism $\sigma([x:y:z:w])=[x:-y:-z:w]$. Then, the quotient $Y=F/<\sigma>$ has the Hilbert Property over $\Q$.
\end{proposition}
\begin{proof}
	For any $d \in \Q^*\setminus (\Q^*)^2$, we note that $F_d$, as defined in Lemma \ref{Cond:final}, is the twist of the Fermat quartic $F:x^4+y^4=z^4+w^4$ by the automorphism $\sigma([x:y:z:w])=[x:-y:-z:w]$, with respect to the quadratic field extension $\Q(\sqrt{d})/\Q$. Hence, by Theorem \ref{teoremino} applied to $X=F$ and $G=\{1,\sigma\}$, in order to prove the proposition, it is sufficient to show that, for each number field $K$, there exists a $d \in \Q^* \setminus (\Q^*)^2$ such that $\Q(\sqrt{d})\nsubseteq K$ and $F_d/\Q$ has the Hilbert Property.
	
%
	
	
	But by Proposition \ref{Pr:diagonalquartics} we know that, for $d \in \Q^*$, $F_d$ has the Hilbert Property as soon as $F_d(\Q)\nsubseteq \Omega_d$.
	
	The statement follows now from Lemma \ref{Cond:final} applied to any rational numbers $a_1,\dots,a_l \in \Q^*$ such that
	\[
	\{\Q(a_1),\dots,\Q(a_l)\}=\{L \subset K:[L:\Q]=1,2 \}.
	\]

\end{proof}

\begin{remark}
	We focused here on proving the Hilbert Property for the specific example of the K3 surface $Y$, since this was proposed by Corvaja and Zannier in \cite[Rmk 3.5]{articoloHP}. However, the technique of the proof of Proposition \ref{Prop:quotient} can be applied to other quotients as well. For instance, it works for quotients of the diagonal quartic surfaces $V_{a,b,c,d}$ which satisfy the hypothesis of Proposition \ref{Pr:diagonalquartics}, by the automorphism $\sigma_{abcd}:[x:y:z:w]\mapsto [x:-y:-z:w]$.
\end{remark}

\bibliographystyle{alpha}      
\bibliography{Non_rational_HP}

\end{document}